\documentclass[a4paper,12pt]{amsart}

\usepackage{graphicx}
\usepackage{amsmath, amsthm, amsfonts, amssymb, amscd, bm}
\usepackage{verbatim}
\usepackage{url}
\usepackage[utf8]{inputenc}
\usepackage[all,cmtip]{xy} 

\newcommand{\complex}{{\mathbb C}}
\newcommand{\cdisc}{{\mathbb D}}
\newcommand{\cstar}{{\mathbb C}^*}
\newcommand{\projspace}[1]{{\mathbb P}^{#1}}
\newcommand{\integers}{{\mathbb Z}}
\renewcommand{\emptyset}{\varnothing}
\newcommand{\scriptO}{{\mathcal O}}
\newcommand{\littlecomp}{{(X\times\complex)\setminus\Gamma_m}}
\newcommand{\bigcomp}{{\complex^{n+1}\setminus\Gamma_{m\circ\psi}}}

\theoremstyle{plain}
\newtheorem{theorem}{Theorem}[section]
\newtheorem{lemma}[theorem]{Lemma}
\newtheorem{proposition}[theorem]{Proposition}
\newtheorem{corollary}[theorem]{Corollary}

\theoremstyle{definition}
\newtheorem{definition}[theorem]{Definition}
\newtheorem{example}[theorem]{Example}

\theoremstyle{remark}
\newtheorem{remark}[theorem]{Remark}

\usepackage[pdftex,
    pdfauthor={Alexander Hanysz},
    pdftitle={Oka properties of some hypersurface complements}
    ]{hyperref} 

\begin{document}

\date{28th November 2011; revised 19th April 2012}
\author{Alexander Hanysz}
\address{School of Mathematical Sciences, University of Adelaide,
    Adelaide SA 5005, Australia}
\email{alexander.hanysz@adelaide.edu.au}
\title{Oka properties of some hypersurface complements}
\begin{abstract}
  Oka manifolds can be viewed as the ``opposite''
  of Kobayashi hyperbolic manifolds.
  Kobayashi asked whether the complement in projective space
  of a generic hypersurface
  of sufficiently high degree is hyperbolic.
  Therefore it is natural to investigate Oka properties
  of complements of low degree hypersurfaces.
  We determine which complements of hyperplane arrangements
  in projective space are Oka.
  A related question is which hypersurfaces
  in affine space have Oka complements.
  We give some results for graphs of meromorphic functions.
\end{abstract}

\maketitle

\section{Introduction}

A complex manifold $X$ is {\em hyperbolic} (in the sense of Kobayashi) if,
informally speaking, there are ``few'' maps $\complex\to X$,
and {\em Oka} if there are ``many'' maps $\complex\to X$,
in a sense to be made precise in Section~\ref{section:Oka} below.
Hyperbolic manifolds have been extensively studied since
the late 1960s.
Oka theory is a more recent development,
motivated by Gromov's paper \cite{Gromov-1989} of 1989;
the definition of an Oka manifold
was only published in 2009, by Forstnerič \cite{Forstneric-2009}.

Many interesting examples of hyperbolic manifolds arise from
complements of projective hypersurfaces.
In particular, Kobayashi asked
\cite[problem~3 on page~132]{Kobayashi-1970}
whether the complement in $\projspace{n}$
of a generic hypersurface of sufficiently high degree
should be hyperbolic.
This has been proved for $n=2$
by Siu and Yeung \cite{Siu-Yeung-1996a},
but is still an open problem in higher dimensions.
The degenerate case of the complement of a finite collection
of hyperplanes is well understood.
In particular, the complement in $\projspace{n}$
of at least $2n+1$ hyperplanes
in general position is hyperbolic,
and the complement of a collection of $2n$ or fewer hyperplanes
is never hyperbolic.
For hyperplanes not in general position,
some necessary conditions for hyperbolicity are known.
See Kobayashi's monograph
\cite[Section~3.10]{Kobayashi-1998} for details.

Since the Oka property can be viewed as a sort of anti-hyperbolicity,
it makes sense to ask which hypersurfaces have Oka complements.
In Section~\ref{section:hyperplane} of this paper
we give a complete answer to this question
for complements of hyperplane arrangements in projective space.
The main result of this section,
Theorem~\ref{thm:hyperplanes}, states
that the complement of $N$ hyperplanes in $\projspace{n}$
is Oka if and only if the hyperplanes are in general position
and $N\leq n+1$.
We also investigate the weaker Oka-type properties
of dominability by $\complex^n$ and $\complex$-connectedness:
in this context we find that a non-Oka complement
also fails to possess these weaker properties.

In Section~\ref{section:graph} we give a sufficient condition
for the complement
of the graph of a meromorphic function to be Oka.
Our Theorem~\ref{thm:mero_general_base} states that
if $m:X\to\projspace{1}$ can be written in the form
$m=f+1/g$ for holomorphic functions $f$ and $g$,
then the graph complement is Oka if and only if $X$ is Oka.
This is motivated by the open problem
of whether the complement in $\projspace{2}$
of a smooth cubic curve is Oka:
given a cubic curve,
there is an associated meromorphic function
and a branched covering map from the graph complement of that function
to the cubic complement.
For details, see Buzzard and~Lu \cite[page 644--645]{Buzzard-Lu-2000}.
We also explore the question of when the decomposition $m=f+1/g$ exists
(Lemma~\ref{lemma:mero_sum}).
For meromorphic functions that cannot be written in this form,
further work is required to understand the Oka properties
of the graph complements.

\noindent\textit{Acknowledgements.}
I thank Finnur Lárusson for
many helpful discussions during the preparation of this paper,
and Franc Forstnerič
for making available preliminary drafts
of his book \cite{Forstneric-2011}.
I am also grateful to the anonymous referee for
constructive suggestions,
in particular the proof of the Oka property
in Example~\ref{ex:not_standard_form}
for the case $\nu=1$.

\section{Oka manifolds and hyperbolic manifolds} \label{section:Oka}

In this section we recall the definitions
of Oka manifolds and hyperbolic manifolds,
and collect some results that will be used later.
For background, motivation and further details of Oka theory, see
the survey article \cite{Forstneric-Larusson-2011}
of Forstneri{\v{c}} and L{\'{a}}russon
and the recently published book \cite{Forstneric-2011} of Forstneri{\v{c}}.
For more on hyperbolicity,
see the monograph \cite{Kobayashi-1998} of Kobayashi.

\begin{definition} \label{def:oka}
A complex manifold $X$ is an \emph{Oka manifold}
if every holomorphic map $K\to X$ from (a neighbourhood of)
a convex compact subset
$K$ of $\complex^n$ can be approximated uniformly on $K$
by holomorphic maps $\complex^n\to X$.
\end{definition}

This defining property is also referred to as the
\emph{convex approximation property (CAP)}.

\begin{definition}
The \emph{Kobayashi pseudo-distance} on a complex manifold $X$
is the largest pseudo-distance such that every holomorphic map
$\cdisc\to X$ is distance-decreasing,
where $\cdisc$ denotes the complex unit disc
with the Poincaré metric.
We say that $X$ is \emph{hyperbolic}
if the Kobayashi pseudo-distance is a distance.
\end{definition}

If $X$ is Oka then the Kobayashi pseudo-distance on $X$ is identically zero;
thus Oka manifolds can be viewed as ``anti-hyperbolic''.
The most fundamental examples of Oka manifolds are complex Lie groups
and their homogeneous spaces; in particular,
$\projspace{n}$ and $\complex^n$ are Oka.
Bounded domains in $\complex^n$ are always hyperbolic.
If $X$ is a Riemann surface,
then $X$ is Oka if and only if it is one of $\complex$,
$\cstar$ (the punctured plane), $\projspace{1}$ or a torus;
otherwise it is hyperbolic.

Every Oka manifold $X$ of dimension $n$ is \emph{dominable}
by $\complex^n$, in the sense that there exists a holomorphic map
$\complex^n\to X$ that has rank~$n$ at some point of $\complex^n$.

Oka manifolds are also $\complex$-connected: every pair of points
can be joined by an entire curve,
i.e.\ for any pair of points there exists a holomorphic map
from $\complex$ into the manifold whose image contains both points.
This property is mentioned by Gromov~\cite[3.4(B)]{Gromov-1989},
and follows easily from the ``basic Oka property''
described in \cite[page~16]{Forstneric-Larusson-2011}.
(The definition of \emph{$\complex$-connected} is not standardised:
the term can also refer to the weaker property
that every pair of points can be joined
by a finite chain of entire curves,
by analogy with the case of rational connectedness.)

In general it is difficult to verify the condition of
Definition~\ref{def:oka} directly.
Instead, sprays (in the sense of Gromov: see below) and fibre bundles
are of fundamental importance.
If $\pi:X\to Y$ is a holomorphic fibre bundle
with Oka fibres, then $X$ is Oka if and only if $Y$ is Oka.
(In fact there is a far more general notion of
an \emph{Oka map} which preserves the Oka property, but this will
not be needed here.)
In particular, products of Oka manifolds are Oka,
and a manifold is Oka if it has a covering space
that is Oka.

\begin{definition} \label{def:spray}
A \emph{spray} over a complex manifold $X$ consists of
a holomorphic vector bundle $\pi:E\to X$
and a holomorphic map $s:E\to X$
such that $s(0_x)=x$ for all $x\in X$.
We say that $s$ is \emph{dominating} at the point $x\in X$
if the differential $ds_{0_x}$ maps the vertical subspace $E_x$ of $T_{0_x}E$
surjectively onto $T_xX$.
A family of sprays $(E_j,\pi_j,s_j)$, $j=1,\ldots,m$,
is \emph{dominating} at $x$ if
$$(ds_1)_{0_x}(E_{1,x})+\cdots+(ds_m)_{0_x}(E_{m,x})=T_xX.$$
The manifold $X$ is \emph{elliptic} if there exists a spray
that is dominating at every point of $X$,
and \emph{weakly subelliptic} if for every compact set $K\subset X$
there exists a finite family of sprays over $X$
that is dominating at every point of $K$.
\end{definition}

The concept of a spray can be viewed as a generalisation
of the exponential map for a complex Lie group:
for example, see \cite[Examples~5.3]{Forstneric-Larusson-2011}
or \cite[Proposition~5.5.1]{Forstneric-2011}.

Every elliptic or weakly subelliptic manifold is Oka.

The following property is equivalent to the CAP.

\begin{definition} \label{def:cip}
A complex manifold $X$ satisfies the
\emph{convex interpolation property (CIP)}
if whenever $T$ is a contractible subvariety of $\complex^m$ for some $m$,
every holomorphic map $T\to X$ extends
to a holomorphic map $\complex^m\to X$.
\end{definition}
(Equivalently, we could take $T$
to be any subvariety of $\complex^m$ that is biholomorphic
to a convex domain in $\complex^n$; hence the use of the word \emph{convex}.)

A useful tool for proving hyperbolicity is
Borel's generalisation of Picard's little theorem.
Kobayashi gives three equivalent formulations
(see \cite[Theorem~3.10.2 on page~134]{Kobayashi-1998}),
of which we only need the following.

\begin{theorem}[Picard--Borel] \label{thm:picard-borel}
Let $g_0,\ldots,g_N$ be nowhere vanishing
holomorphic functions on $\complex$, and suppose
$$g_0+\cdots+g_N=0.$$
Partition the index set $\{0,1,\ldots,N\}$
into subsets, putting two indices
$j$ and $k$ into the same subset
if and only if $g_j/g_k$ is constant.
Then for each subset $J$,
$$\sum_{j\in J}g_j=0.$$
\end{theorem}

\begin{remark} \label{remark:borel-picard}
Since the $g_j$ are nowhere vanishing,
it follows that each subset must have size greater than 1.
In particular, if $N=2$ then the partition has only one part,
hence $g_0$, $g_1$ and~$g_2$ are constant multiples of each other.
\end{remark}

\section{Hyperplane complements} \label{section:hyperplane}

Let $F_1,\ldots,F_N$ be nonzero homogeneous linear forms in $n+1$ variables.
We say that the hyperplanes in $\projspace{n}$ defined by the equations
$F_j=0$, $j=1,\ldots,N$, are in \emph{general position}
if every subset of $\{F_1,\ldots,F_N\}$ of size at most $n+1$ is
linearly independent.
If $N\leq n+1$, then a set of $N$ hyperplanes
is in general position if and only if coordinates can be chosen
on $\projspace{n}$ so that the given hyperplanes
are the coordinate hyperplanes $x_j=0$, $j=0,\ldots,N-1$.

\begin{theorem} \label{thm:hyperplanes}
Let $H_1,\ldots,H_N$ be distinct hyperplanes in $\projspace{n}$.
Then the complement $X=\projspace{n}\setminus\bigcup_{j=1}^N H_j$ is Oka
if and only if the hyperplanes are in general position and $N\leq n+1$.
Furthermore, if $X$ is not Oka
then it is not dominable by $\complex^n$
and not $\complex$-connected.
\end{theorem}

Before proving this, we state and prove a sharper form of
Theorem~3.10.15 of Kobayashi's book \cite[page~142]{Kobayashi-1998}.
To state the theorem, it is convenient to introduce the following terminology.

\begin{definition}
Let $H_1,\ldots,H_k$ be distinct hyperplanes in
$\projspace{n}$ defined by linear forms $F_1,\ldots,F_k$,
and suppose the forms satisfy a minimal linear relation
of the form
$$c_1F_1+\cdots+c_kF_k=0$$
where $c_j\neq0$ for all $j$.
(By ``minimal'' we mean that
$\sum_{j\in J}c_j F_j\neq0$ for every proper nonempty subset
$J$ of $\{1,\ldots,k\}$.)
Then the \emph{diagonal hyperplanes}
of the linear relation
are the hyperplanes defined by the linear forms
$\sum_{j\in J}c_j F_j$
where $J$ is a subset of $\{1,\ldots,k\}$
with $2\leq |J| \leq k-2$.
(If $k\leq3$, there are no diagonal hyperplanes.)
The \emph{associated subspaces}
of $\{H_1,\ldots,H_k\}$
are the linear subspaces of $\projspace{n}$ which contain
$\bigcap_{j=1}^k H_j$
with codimension 1.
(If $\bigcap H_j=\emptyset$, the associated subspaces
are exactly the points of $\projspace{n}$.)
\end{definition}

\begin{remark}
If $p\in\projspace{n}\setminus\bigcap H_j$,
then $p$ is contained in exactly one associated subspace
for each minimal linear relation.
\end{remark}

\begin{example}
On $\projspace{2}$ consider the linear forms
\begin{align*}
    F_1&=x_1, \\
    F_2&=x_2, \\
    F_3&=x_1-x_0, \\
    F_4&=x_2-x_0.
\end{align*}
If we consider $x_0=0$ to be the line at infinity,
then the lines $F_j=0$, $j=1,2,3,4$,
are the sides of a ``unit square'' in the affine plane.
The linear relation $F_1-F_2-F_3+F_4=0$
has three diagonal lines
(noting that $J=\{1,2\}$ and $J=\{3,4\}$ give the same line,
and so on).
They are the two diagonals of the square
($x_1=x_2$ and $x_1+x_2=x_0$),
and the line at infinity ($x_0=0$).
\end{example}

\begin{example}
Let $P$ be any point of $\projspace{2}$
and let $F_1$, $F_2$ and $F_3$
be linear forms defining three distinct lines through $P$.
Then there exists a linear relation among $F_1$, $F_2$ and $F_3$,
and the associated subspaces are
the lines through~$P$.
\end{example}

\begin{theorem}\label{thm:subspaces}
Let $H_1,\ldots,H_N$ be distinct hyperplanes in
$\projspace{n}$ defined by linear forms $F_1,\ldots,F_N$,
and let $f:\complex\to\projspace{n}\setminus\bigcup H_j$
be a holomorphic map.
Suppose that $F_1,\ldots,F_N$ are linearly dependent.
Then for each subset of $F_1,\ldots,F_N$ satisfying a minimal linear relation,
there is a diagonal hyperplane or an associated subspace containing
the image of~$f$.
\end{theorem}

\begin{remark}
In the case where $\bigcap H_j=\emptyset$,
the associated subspaces are points,
so the conclusion is that either $f$ is constant
or the image is contained in a diagonal hyperplane.
\end{remark}

\begin{proof}[Proof of Theorem~\ref{thm:subspaces}]
First we note that $f$ can be lifted
to a holomorphic map $\tilde f:\complex\to\complex^{n+1}\setminus\{0\}$.
To see this, observe that the quotient map
$\complex^{n+1}\setminus\{0\}\to\projspace{n}$
can be regarded as the universal line bundle over $\projspace{n}$
with the zero section removed.
Thus lifting $f$ is equivalent to finding
a nowhere vanishing section of the pullback by $f$
of the universal bundle.
But the vanishing of the cohomology group $H^1(\complex,\scriptO^*)$
guarantees that line bundles over $\complex$ are trivial,
and therefore a nowhere vanishing section always exists.

By reordering and rescaling the defining forms,
we can put a minimal linear relation
in the form
$$F_1+\cdots+F_k=0.$$

Define entire functions $h_1,\ldots,h_k$ by
$h_j=F_j\circ \tilde f$.
Then the $h_j$ satisfy the hypotheses of
the Picard--Borel theorem (Theorem~\ref{thm:picard-borel}):
each $h_j$ vanishes nowhere
(because the image of $f$ misses all the hyperplanes),
and the $h_j$ sum to the zero function
(because of the linear relation between the $F_j$).
Theorem~\ref{thm:picard-borel}
tells us that there is a subset
$J\subset\{1,\ldots,k\}$
with
$$\sum_{j\in J}h_j=0$$
and such that all the ratios $h_\mu/h_\nu$
are constant for $\mu,\nu\in J$.
There are two possibilities.

First, if $J$ is a proper subset of $\{1,\ldots,k\}$
then $J$ must have size
at least 2 and at most $k-2$.
(If $J$ either consisted of or omitted only a singleton $j$,
then the corresponding $h_j$ would be identically zero.)
In this case the linear form
$$F=\sum_{j\in J}F_j$$
defines a diagonal hyperplane in $\projspace{n}$.
(The minimality of the linear relation
implies that $F$ is nonzero.)
The image of $f$ lies in this hyperplane.

The second case is that $J=\{1,\ldots,k\}$.
Then there exist nonzero constants $c_1\ldots,c_{k-1}$
such that $$h_j=c_jh_k$$ for $j=1,\ldots,{k-1}$.
This means that the image of $\tilde f$
lies in each of the hyperplanes $F_j=c_jF_k$.
Let $A$ and $B$ be
the linear subspaces of $\complex^{n+1}$ given by
\begin{align*}
  A &= \bigcap_{j=1}^{k-1}\{F_j-c_jF_k=0\}, \\
  B &= \bigcap_{j=1}^k \{F_j=0\}.
\end{align*}
Clearly $B\subset A$.
It remains to show that $B$ has codimension at most 1 in $A$.
Equivalently, we wish to show that given $x,y\in A\setminus B$,
some nontrivial linear combination of $x$ and $y$ lies in $B$.
The numbers $\alpha=F_k(x)$
and $\beta=F_k(y)$ are both nonzero.
Then $F_k(\beta x-\alpha y)=0$,
so $F_j(\beta x-\alpha y)=0$ for all $j$,
hence $\beta x-\alpha y\in B$.
\end{proof}

\begin{remark}
In the above proof, the subspaces $A$ and $B$ can never be equal
(because the image of $f$ misses all of the $H_j$).
A naive dimension argument might suggest that $A$ and $B$ have the same dimension.
However, the fact that the $h_j$ sum to zero
implies the relation $c_1+\cdots+c_{k-1}=-1$,
so the forms $F_j-c_jF_k$ are linearly dependent:
their sum is zero.
\end{remark}

\begin{remark}
In the case $J=\{1,\ldots,k\}$,
Kobayashi states that $f$ is constant
(\cite[page~142, proof of Theorem~3.10.15, last paragraph]{Kobayashi-1998}).
In fact there exist nonconstant maps whose images lie in associated subspaces,
but this does not invalidate the conclusion
of Kobayashi's Theorem~3.10.15.
For an example of such a map,
take $n=4$ with linear forms
$$F_1=x_0-x_2,\quad F_2=x_2-x_1,\quad F_3=x_1-x_3,\quad F_4=x_3-x_0,\quad F_5=x_4$$
so that $F_1+F_2+F_3+F_4=0$.
Let $f:\complex\to X$ be a function that lifts to
$$\tilde f(t)=(t,t+1,t+2,t+3,1).$$
Then $h_1,\ldots,h_4$ are the constant functions $-2,1,-2,3$
respectively, but $f$ is not constant.
\end{remark}

\begin{corollary}
Let $H_1,\ldots,H_N$ be distinct hyperplanes
in $\projspace{n}$,  not in general position.
Let $p$ be any point of $X=\projspace{n}\setminus\bigcup H_j$.
Then there is a finite collection of proper subspaces
of $T_p X$ with the property that for every map $f:\complex\to X$
with $f(0)=p$,
the derivative $df(0)$ lies in one of those subspaces.
\end{corollary}

\begin{proof}
By Theorem~\ref{thm:subspaces},
the image of $f$ is restricted to a proper linear subspace of $X$.
Thus there is a corresponding subspace of $T_p X$
containing $df(0)$.
We just need to verify that there are only finitely many possible subspaces.
But each point of $X$ is contained in exactly one
associated subspace for each minimal linear relation among the $F_j$,
and there are only finitely many diagonal hyperplanes.
\end{proof}

\begin{corollary} \label{cor:notdom}
If the distinct hyperplanes $H_1,\ldots,H_N$ in $\projspace{n}$
are not in general position,
then $\projspace{n}\setminus\bigcup H_j$ is not dominable by $\complex^n$.
\end{corollary}

\begin{proof}
Let $f$ be a map from $\complex^n$ into $\projspace{n}\setminus\bigcup H_j$,
with $f(0,\ldots,0)=p$.
The image of $df(0)$ is spanned by the $n$ vectors
$$d(t\mapsto f(t e_j))|_{t=0}\quad(j=1,\ldots,n)$$
where $\{e_1,\ldots,e_n\}$ is a basis for $\complex^n$.
If $df(0)$ is surjective,
those vectors are linearly independent,
so there will be no finite set of proper subspaces
containing
$$d(t\mapsto f(tv))|_{t=0}$$
for all $v\in\complex^n$, contradicting the previous corollary.
\end{proof}

\begin{proof}[Proof of Theorem~\ref{thm:hyperplanes}]
Write $X$ for the space $\projspace{n}\setminus\bigcup_{j=1}^N H_j$.

\textit{Case 1: hyperplanes in general position and $N>n+1$.}
In this case, Kobayashi's Theorem~3.6.10 \cite[page~136]{Kobayashi-1998}
tells us that the image of a nonconstant holomorphic map $\complex\to X$
must lie in one of a finite collection of hyperplanes.
Therefore $X$ is not dominable by $\complex^n$.
Also, $X$ is not $\complex$-connected: distinct points
outside the finite collection of hyperplanes
cannot be joined by an entire curve.

\textit{Case 2: hyperplanes in general position and $N\leq n+1$.}
If $N=0$, then $X=\projspace{n}$ is Oka.
For $N>0$, the fact that the hyperplanes are in general position
means that we can choose coordinates
so that $H_j$ is the hyperplane $x_{j-1}=0$ for $j=1,\ldots,N$.
Then we see that
$X\cong\cstar\times\cdots\times\cstar\times\complex\times\cdots\times\complex$
with $N-1$ factors $\cstar$ and $n+1-N$ factors $\complex$.
This is a product of Oka manifolds, hence Oka.

\textit{Case 3: hyperplanes not in general position.}
The fact that $X$ is not dominable, and therefore not Oka,
is just Corollary~\ref{cor:notdom} above.
To see that $X$ is not $\complex$-connected,
choose any point $p\in X$.
Then Theorem~\ref{thm:subspaces} gives a finite
collection of hyperplanes containing every entire curve through $p$.
If $q$ is a point of $X$ outside that finite collection,
then $p$ and $q$ cannot be joined by an entire curve.
\end{proof}

\begin{remark}
In fact if $X$ is not Oka,
then it does not satisfy the weaker version of $\complex$-connectedness
referred to above: there exist pairs of points that cannot
be connected even by a finite chain of entire curves.
In case 1 of the above proof this is immediate.
For case 3, some further work is needed, since
the finite collection of hyperplanes referred to
can vary with the choice of the point $p$.
The key ingredients are the fact that there are only
finitely many diagonal hyperplanes in total,
and that given an associated subspace $A$
and a diagonal hyperplane $D$ of the same configuration,
either $A\subset D$ or $A\cap D\subset\bigcup H_j$.
In other words, points inside a diagonal hyperplane
cannot be joined to points outside via associated subspaces.
\end{remark}

\section{Complements of graphs of meromorphic functions} \label{section:graph}

Buzzard and Lu \cite[Proposition~5.1]{Buzzard-Lu-2000}
showed that the complement in $\projspace{2}$
of a smooth cubic curve is dominable by $\complex^2$.
Their method of proof was to construct
a meromorphic function associated with the cubic,
and a branched covering map
from the complement of the graph of that function to the complement of the cubic,
and then show that the graph complement is dominable.
We will show that the graph complement is in fact Oka;
this result can be generalised to meromorphic functions on Oka manifolds
other than $\complex$, subject to an additional hypothesis.
(Note that our result is not enough to settle the question
of whether the cubic complement is Oka.
We know that the Oka property passes down through \emph{unbranched} covering maps,
but no similar result is known for branched coverings.)

For a holomorphic map $m:X\to\projspace{1}$ on a complex manifold $X$,
that is to say either a
meromorphic function with no indeterminacy
or else the constant function $\infty$,
we will write $\Gamma_m$ for the affine graph
$$\Gamma_m=\{(x,m(x))\in X\times\complex:m(x)\neq\infty\}.$$
This is a closed subset of $X\times\complex$,
so the set $(X\times\complex)\setminus\Gamma_m$ is a complex manifold.
(If $m$ is identically $\infty$, then $\Gamma_m$ is the empty set.)

Buzzard and Lu's result relies on the fact that meromorphic functions on $\complex$
can be written in the following form.

\begin{lemma} \label{lemma:BL}
For every meromorphic function $m:\complex\to\projspace{1}$
there exist holomorphic functions $f,g:\complex\to\complex$ such that
$$m=f+\frac{1}{g}.$$
\end{lemma}

In other words, the projection map from $\complex^2\setminus\Gamma_m$
onto the first coordinate has a holomorphic section
given by $x\mapsto(x,f(x))$.

The result follows from the classical theorems of Mittag-Leffler
and Weierstrass;
see \cite[page~645]{Buzzard-Lu-2000} for details.

The analogous result in higher dimensions is not true.
We have the following topological criterion.

\begin{lemma} \label{lemma:mero_sum}
Let $m$ be a holomorphic map from $\complex^n$ to $\projspace{1}$,
and write $m=h/k$ for holomorphic functions
$h,k:\complex^n\to\complex$ with no common zeros.
Then the following statements are equivalent.
\begin{enumerate}
    \item There exist holomorphic functions $f,g:\complex^n\to\complex$
    such that $$m=f+\frac{1}{g}.$$
    \item The function $h$ has a logarithm on the zero set $Z(k)$ of $k$.
    \item The function $h$ has a logarithm on a neighbourhood of $Z(k)$.
\end{enumerate}
\end{lemma}

\begin{proof}
$(1\Rightarrow 2)$:
The function
$$\frac{1}{g}=m-f=\frac{h-kf}{k}$$
has no zeros.
Therefore $h-kf$ is a nowhere vanishing entire function, so
$$h-kf=e^\mu$$
for some holomorphic $\mu:\complex^n\to\complex$.
Then $\mu|_{Z(k)}$ is the desired logarithm of $h$.

\bigskip
$(2\Rightarrow 3)$:
Suppose there is a holomorphic function $\lambda$
on $Z(k)$ such that $e^\lambda=h|_{Z(k)}$.
We wish to find a neighbourhood $U$ of $Z(k)$,
small enough that $h$ vanishes nowhere on $U$,
such that the inclusion map $Z(k)\hookrightarrow U$
induces an epimorphism of fundamental groups.
Given such a neighbourhood, we have the following situation.
$$ \xymatrix@R=15mm@C=20mm{
    & & \complex \ar[d]^{\text{exp}} \\
    Z(k) \ar@{^{(}->}[r] \ar@<1mm>[urr]^\lambda
    & U \ar[r]^h \ar@{.>}[ur]_{\lambda'} & ~ \cstar
} $$
Then the existence of $\lambda$, together with the epimorphism
$\pi_1(Z(k))\to\pi_1(U)$,
tells us that $h|_U$ satisfies the lifting criterion
for the covering map $\exp:\cstar\to\complex$.
Therefore there exists a holomorphic function $\lambda'$
such that $e^{\lambda'}=h$ on $U$.

To find a suitable neighbourhood $U$,
we start by realising $\complex^n$ as a simplicial complex
with $Z(k)$ as a subcomplex.
(The existence of such a simplicial complex is guaranteed
by standard results on the topology of subanalytic varieties:
see for example \cite[Theorem~1]{Lojasiewicz-1964}.)
Then we can find a basis of neighbourhoods of $Z(k)$
such that $Z(k)$ is a strong deformation retract of each basis set
(this is a general fact about CW-complexes:
see \cite[Prop.~A.5, p.~523]{Hatcher-2002}).
Finally, we choose a basis set $U$ small enough
that $h$ does not vanish on $U$.

\bigskip
$(3\Rightarrow 1)$:
First consider the situation of hypothesis~(2):
suppose $\lambda:Z(k)\to\complex$ is a logarithm for $h$.
We wish to find a suitable
holomorphic function $\mu:\complex^n\to\complex$
which extends $\lambda$.
Then we can define
$$f=\frac{h-e^\mu}{k}\text{ and }g=\frac{k}{e^\mu}.$$
In order for such $f$ to be a well defined holomorphic function,
we require that $h-e^\mu$ should vanish on $Z(k)$ to order
at least the order of vanishing of $k$,
in other words that the divisors should satisfy
$$(h-e^\mu)\geq(k).$$

By hypothesis, we in fact have a neighbourhood $U$ of $Z(k)$
and a logarithm $\lambda':U\to\complex$
for $h$ on $U$.
Then the above condition is equivalent to requiring that
$\mu$ agrees with $\lambda'$ on $Z(k)$
up to the order of vanishing of $k$.

We write $\scriptO$ for the sheaf
of germs of holomorphic functions on $\complex^n$.
Consider the exact sequence of sheaves
$$0\to k\scriptO\to\scriptO\to\scriptO/k\scriptO\to 0.$$
The sheaf $k\scriptO$ is coherent, so by
Cartan's Theorem~B,
$H^1(\complex^n,k\scriptO)=0$,
and therefore the map
$\scriptO(\complex^n)\to(\scriptO/k\scriptO)(\complex^n)$
is surjective.
Noting that the stalk of $\scriptO/k\scriptO$ at any point
outside $Z(k)$ is zero,
we see that the function $\lambda'$
represents an element of $(\scriptO/k\scriptO)(\complex^n)$.
Then we can choose $\mu$ to be any preimage of that element.
\end{proof}

\begin{remark}
When $n=1$ in the above lemma,
the zero set of $k$ is discrete,
and so a logarithm of $h$ always exists on $Z(k)$.
Thus Lemma~\ref{lemma:BL} is a special case
of Lemma~\ref{lemma:mero_sum}.
\end{remark}

\begin{remark}
The only properties of $\complex^n$ used in the above proof
are that it is Stein and simply connected
and that all meromorphic functions on $\complex^n$ can be written as a quotient.
Thus we can generalise the result:
if $X$ is a simply connected Stein manifold, $h,k\in\scriptO(X)$
have no common zeros and $m=h/k$,
then the three statements given in the lemma are equivalent.
\end{remark}

\begin{example} \label{ex:not_standard_form}
For positive integers $\nu$, the functions $m_\nu:\complex^2\to\projspace{1}$
given by
$$m_\nu(x,y)=\frac{x}{x y^\nu -1}$$
cannot be written in the form $f+1/g$.
At present it is not known whether any of the spaces
$\complex^3\setminus\Gamma_{m_\nu}$ for $\nu\geq2$ are Oka.
In the case $\nu=1$, the Oka property for $\complex^3\setminus\Gamma_{m_1}$
follows from work of Ivarsson and
Kutzschebauch~\cite[Lemmas~5.2 and~5.3]{Ivarsson-Kutzschebauch-2012}.
Specifically, let $p$ be the polynomial
$p(x,y,z)=xyz-x-z$.
Then $\Gamma_{m_1}$ is the level set $p^{-1}(0)$,
and the complement is the union of all the other level sets.
The complement is isomorphic to the product of $\cstar$
with the level set $p^{-1}(1)$ via the map
$\cstar\times p^{-1}(1)\to\complex^3\setminus\Gamma_{m_1}$ given by
$$(\lambda,x,y,z)\mapsto(\lambda x, \lambda^{-1} y, \lambda z).$$
Now $p^{-1}(1)$ is smooth,
and by the results of Ivarsson and~Kutzschebauch,
its tangent bundle is spanned by
finitely many complete holomorphic vector fields.
This implies that the set is Oka
(see for example \cite[Example~5.5.13(B)]{Forstneric-2011}),
so $\cstar\times p^{-1}(1)$ is Oka.
\end{example}

With these considerations in mind,
we are ready to state the main result of this section.

\begin{theorem} \label{thm:mero_general_base}
Let $X$ be a complex manifold, and
let $m:X\to\projspace{1}$
be a holomorphic map.
Suppose $m$ can be written in the form
$m=f+1/g$ for holomorphic functions $f$ and~$g$.
Then $\littlecomp$ is Oka if and only if $X$ is Oka.
\end{theorem}

\begin{remark} \label{rem:section}
The existence of the decomposition $m=f+1/g$ is a geometric
condition that is of some independent interest:
it is equivalent to the condition that the projection
map from $\littlecomp$ onto the first factor has a holomorphic section.
The projection map is an elliptic submersion
in the sense defined in \cite[page 24]{Forstneric-Larusson-2011}.
(It is easy to see that it is a \emph{stratified} elliptic submersion,
as defined in \cite[page 25]{Forstneric-Larusson-2011}.
For a sketch of why it is an (unstratified) elliptic submersion,
see Remark~\ref{rem:elliptic} below.)
However, unless either $m$ has no poles or $m=\infty$,
the projection is not an Oka map,
because it is not a topological fibration.

In the case where $X$ is Stein, it follows from
\cite[Theorem~5.4~(iii)]{Forstneric-Larusson-2011}
that the existence of a continuous section of the elliptic submersion
implies the existence of a holomorphic section.
For general $X$, one might expect that this ellipticity property
could be applied to yield a simpler proof than the one presented below.
So far, such a proof has been elusive.
\end{remark}

We will first prove Theorem~\ref{thm:mero_general_base}
for the special case $X=\complex^n$,
and then show how the convex interpolation property (Definition~\ref{def:cip})
for general $X$ reduces to the special case.
The proof for $X=\complex^n$ involves a variation
of Gromov's technique of localisation of algebraic subellipticity
(see \cite[Lemma~3.5B]{Gromov-1989}
and \cite[Proposition~6.4.2]{Forstneric-2011}).
This relies on the following lemma.

\begin{lemma} \label{lemma:localise}
Let $g:\complex^n\to\complex$ be a holomorphic function,
not identically zero,
and suppose $x_0\in\complex^n$ satisfies $g(x_0)=0$.
Then for all $s\in\complex^n$,
$$\lim_{\substack {x\to x_0 \\ g(x)\neq0}}\left(
  \frac{1}{g(x)}-\frac{1}{g(x+g(x)^2s)}\right)
  =g'(x_0)(s).
$$
\end{lemma}

\begin{proof}
First, in order for the limit to make sense,
we need to verify that $x_0$ has a neighbourhood
on which $g(x+g(x)^2s)$ vanishes only when $g(x)$ vanishes.
We use the approximation
\begin{align} 
  g(x+h)&=g(x)+g'(x)(h)+O(|h|^2). \label{eq:approx} \\
\intertext{With $h=g(x)^2s$, this gives}
  g(x+g(x)^2s)&=g(x)+g'(x)(g(x)^2s)+O(|g(x)^4|). \notag \\
\intertext{When $x\neq x_0$ and $g(x)\neq0$,
    if $x$ is close to $x_0$, then the second and third terms of the right hand side
    are much smaller than the first,
    so $g(x+g(x)^2s)\neq0$, as required.
\newline\indent
    Now, using \eqref{eq:approx} again, we obtain
} 
  \frac{1}{g(x)}-\frac{1}{g(x+h)}
      &= \frac{g(x+h)-g(x)}{g(x)g(x+h)} \notag \\
      &= \frac{g'(x)(h)+O(|h|^2)}{g(x)(g(x)+g'(x)(h)+O(|h|^2))}.\notag \\
\intertext{(In the event that $g(x+h)$ vanishes, we interpret
    the fractions as meromorphic functions.)   
    Replacing $h$ with $g(x)^2 s$,
    and using the fact that $g'(x)$ is a linear map, gives
} 
  \frac{1}{g(x)}-\frac{1}{g(x+g(x)^2 s)}
      &=\frac{g'(x)(g(x)^2 s)+O(|g(x)|^4)}{g(x)(g(x)+g'(x)(g(x)^2 s)+O(|g(x)|^4))} \notag \\
      &=\frac{g'(x)(s)+O(|g(x)|^2)}{1+g(x)g'(x)(s)+O(|g(x)|^3)}. \notag
\end{align}
As $x\to x_0$ this expression tends to $g'(x_0)(s)$.
\end{proof}

\begin{remark}
The exponent 2 in the lemma corresponds to a doubly twisted line bundle
in the proof of Proposition~\ref{prop:Cn} below.
A single twist would not be sufficient:
for example, if we take $n=1$ and $g(x)=x$
then for $s\neq0$ the expression
$\dfrac{1}{g(x)}-\dfrac{1}{g(x+g(x)s)}$
does not have a finite limit as $x\to0$.
\end{remark}

\begin{proposition} \label{prop:Cn}
Let $m:\complex^n\to\projspace{1}$ be a holomorphic map,
and suppose $m$ can be written in the form
$m=f+1/g$ for holomorphic functions $f$ and $g$.
Then $\complex^{n+1}\setminus\Gamma_m$ is Oka.
\end{proposition}

\begin{proof}
If $g=0$, so that $m=\infty$,
then $\Gamma_m=\emptyset$,
so $\complex^{n+1}\setminus\Gamma_m=\complex^{n+1}$ is Oka.
For the rest of the proof, assume that $g\neq0$.

We will write points of $\complex^{n+1}$ as $(x,y)$ or $(s,t)$
where $x,s\in\complex^n$ and $y,t\in\complex$.
Let $X$ denote the complement of the graph of $1/g$; i.e.\
$$X=\{(x,y)\in\complex^{n+1}:g(x)y\neq 1\}.$$
The map $X\to\complex^{n+1}\setminus\Gamma_m$
given by $(x,y)\mapsto(x,y+f(x))$ is a biholomorphism.
Hence it suffices to prove that $X$ is Oka.

We begin by describing a covering space $Y$ of $X$.
Then we shall exhibit sprays on trivial bundles
over certain subsets of the covering space.
Finally, these sprays will be extended to sprays
on twisted bundles over $Y$, using the above lemma.
(This is the localisation step referred to above.)
This will be sufficient to establish that
$Y$ is weakly subelliptic, hence Oka.
Therefore $X$ is Oka.

The covering space $Y$ is constructed as follows.
Define an equivalence relation $\sim$ on $\complex^{n+1}\times\integers$ by
\begin{multline} \label{eq:equiv}
    (x,y,k)\sim(x',y',k')\text{ if }
    x=x',\> g(x)\neq0\text{ and } \\
    g(x)(y-y')=(k-k')2\pi i.
\end{multline}
Then $Y$ is the quotient space $(\complex^{n+1}\times\integers)/\sim$.
From now on we will write $[x,y,k]$ as shorthand for the equivalence
class in $Y$ of $(x,y,k)$,
and $Y_k$ for the $k$th ``layer'' $\{[x,y,k]:(x,y)\in\complex^{n+1}\}$.

Note that $Y$ can be described in concrete terms
as a hypersurface in $\complex^{n+2}$:
see Remark~\ref{rem:concrete}.
The description of $Y$ used here is chosen to emphasise
the simple form of the sprays $\sigma_k$ described below.

It is straightforward to verify that $Y$ is a Hausdorff space.
We can map each $Y_k$ bijectively to $\complex^{n+1}$ by sending $[x,y,k]$ to $(x,y)$.
Thus $Y$ has the structure of an $(n+1)$-dimensional complex manifold.

By way of motivation for this construction,
observe that if $x_0\in\complex^n$ with $g(x_0)\neq0$,
then the set $\{[x,y,k]\in Y:x=x_0\}$ is a copy of $\complex$,
whereas if $g(x_0)=0$, then $\{[x,y,k]\in Y:x=x_0\}$
is a countable union of disjoint copies of $\complex$.
The covering map described below looks like
an exponential map when $g\neq0$, but the identity map when $g=0$.
The construction involves a holomorphically varying family of holomorphic maps
which include both exponentials and the identity.

We follow Buzzard and Lu \cite[page~645]{Buzzard-Lu-2000} in defining
a function $\phi$ on $\complex^2$ by
\begin{align}
    \phi(x,y)&=\left\{
    \begin{array}{cl}
        \dfrac{e^{xy}-1}{x} & \text{if }x\neq0 \label{eq:BLmap} \\
        y & \text{if }x=0
    \end{array}
    \right. \\
    &= y+\frac{xy^2}{2}+\frac{x^2y^3}{3!}+\cdots. \nonumber
\end{align}
From the series expansion we see that $\phi$ is holomorphic.
Then we define $\pi:Y\to X$ by
$$\pi[x,y,k]=(x,-\phi(g(x),y)).$$
If $(x,y)$ is a point of $X$ with $g(x)=0$,
then the fibre over $(x,y)$ is the set
$$\pi^{-1}(x,y)=\{[x,-y,k]:k\in\integers\}.$$
If $g(x)\neq0$, then a set of unique representatives for
$\pi^{-1}(x,y)$ is given by
$$\{[x,\frac{\log(1-g(x)y)}{g(x)},0]\}$$
for all possible branches of the logarithm.
It follows that all fibres of $\pi$ are isomorphic to $\integers$.
It can be verified that every point of $X$ has
a neighbourhood that is evenly covered by $\pi$,
and so $\pi$ is a covering map.

For each layer $Y_k$ of $Y$
there is a dominating spray $\sigma_k$ on the trivial bundle
$Y_k\times\complex^{n+1}$, given by
\begin{equation*}
\sigma_k([x,y,k];s,t)=[x+s,y+t,k].
\end{equation*}
We wish to construct a bundle $E_k$ over $Y$
and a spray $\tilde\sigma_k:E_k\to Y$
such that $\tilde\sigma_k$ agrees with $\sigma_k$
with respect to a trivialisation of $E_k|_{Y_k}$.
Since every compact subset of $Y$ is covered by finitely many $Y_k$,
this will establish that $Y$ is weakly subelliptic
(Definition~\ref{def:spray}).

To simplify the notation, we will only describe $E_0$ and $\tilde\sigma_0$;
the construction for $k\neq0$ is similar.
Define open subsets $U_1$ and $U_2$ of $Y$ by
\begin{align*}
    U_1&=\{[x,y,k]:k\neq0\}, \\
    U_2&=\{[x,y,k]:k=0\}=Y_0.
\end{align*}
As each $[x,y,k]$ is an equivalence class,
these sets are not in fact disjoint.
(This is the only part of the proof
where the assumption $g\neq0$ is required.)
The intersection $U_1\cap U_2$
is the set of points $[x,y,0]$ with $g(x)\neq0$.
The bundle $E_0$ is described by local trivialisations
$E_0|_{U_\alpha}\to U_\alpha\times\complex^{n+1}$,
$\alpha=1,2$,
with transition map
$\theta_{12}:(U_1\cap U_2)\times\complex^{n+1}\to (U_1\cap U_2)\times\complex^{n+1}$
given by
\begin{equation} \label{eq:transition}
    \theta_{12}([x,y,0];s,t)=([x,y,0];g(x)^2 s,t).
\end{equation}
\noindent
Define $\tilde\sigma_0$ by
\begin{align*}
\tilde\sigma_0|_{U_1}([x,y,k];s,t)
  &= 
  \left\{
    \begin{array}{ll}
        [x+g(x)^2 s,y-k2\pi i/g(x)+t,0] & \text{if }g(x)\neq0, \\
        {}[x ,y-k2\pi i g'(x)(s)+t,k] & \text{if }g(x)=0,
    \end{array}
   \right. \\
\tilde\sigma_0|_{U_2}([x,y,0];s,t)
  &= \sigma_0([x,y,0];s,t) = [x+s,y+t,0].
\end{align*}
The fact that $\tilde\sigma_0|_{U_1}$ is continuous
follows from Lemma~\ref{lemma:localise} together with
equation (\ref{eq:equiv}).
It is easy to verify from equations (\ref{eq:equiv}) and (\ref{eq:transition})
that $\tilde\sigma_0|_{U_1}$ and $\tilde\sigma_0|_{U_2}$
agree on $U_1\cap U_2$.
Thus $\tilde\sigma_0$ is a well defined holomorphic map
from $E_0$ to $Y$ extending $\sigma_0$.
Finally, $\tilde\sigma_0([x,y,k];0,0)=[x,y,k]$,
so $\tilde\sigma_0$ is a spray.
This completes the proof.
\end{proof}

\begin{remark} \label{rem:concrete}
The covering space $Y$ from the above proof
can be embedded into $\complex^{n+2}$
by the map
$[x,y,k]\mapsto(x,-\phi(g(x),y),g(x)y+2\pi ik)$.
The image of this map is the set
    $$Z=\{(x,y,z)\in\complex^n\times\complex\times\complex:
    1-g(x)y=e^z\}$$
and a covering map $Z\to X$ is given by $(x,y,z)\mapsto(x,y)$.
\end{remark}

\begin{proof}[Proof of Theorem \ref{thm:mero_general_base}]
We will write $\pi_1$ and $\pi_2$ for
the projections of the complement $\littlecomp$ onto
$X$ and $\complex$ respectively.
The map $\sigma:X\to\littlecomp$ given by
$\sigma(x)=(x,f(x))$ is a holomorphic section of $\pi_1$.

First suppose $\littlecomp$ is Oka.
The convex interpolation property for $X$
can easily be verified as follows.
Let $\phi:T\to X$ be a holomorphic map from a contractible subvariety
$T$ of some $\complex^n$.
Then by the CIP of $\littlecomp$,
the composite map $\sigma\circ\phi:T\to\littlecomp$
has a holomorphic extension $\psi:\complex^n\to\littlecomp$.
The composition $\pi_1\circ\psi$ is a map $\complex^n\to X$
extending $\phi$.
Therefore $X$ is Oka.

Conversely, suppose $X$ is an Oka manifold, and $m:X\to\projspace{1}$ a holomorphic map
with $m=f+1/g$
as in the statement of the theorem.
We will verify the CIP for $\littlecomp$.

Suppose $T$ is a contractible subvariety of $\complex^n$ for some $n$,
and let $\phi:T\to \littlecomp$
be a holomorphic map.
We want to find a holomorphic map
$\mu:\complex^n\to\littlecomp$
which extends $\phi$.

First of all we can use the CIP for $X$ to extend
the composite map $\pi_1\circ\phi$
to a holomorphic map $\psi:\complex^n\to X$.
This is indicated in the following diagram.

$$  
\xymatrix{
    T \ar@{^{(}->}[dd]_\iota \ar[r]^-\phi
        & (X\times\complex)\setminus\Gamma_m \ar[dd]^{\pi_1} \ar[r]^-{\pi_2}
        & \complex \\
    \\
    \complex^n \ar@{.>}[uur]^\mu \ar[r]^\psi
        & X \ar[r]^m & \projspace{1}}
$$

Now we have a holomorphic map $m\circ\psi:\complex^n\to\projspace{1}$.
In fact $m\circ\psi=f\circ\psi+1/(g\circ\psi)$,
so we know by Proposition~\ref{prop:Cn} that $\bigcomp$ is Oka.
We want to map $T$ into $\bigcomp$,
then use the CIP to extend this map.

Define $\alpha:T\to\complex^{n+1}$ by
$$\alpha(x)=(\iota(x),\pi_2(\phi(x))).$$
Since $\phi(x)$ is an element of $\littlecomp$,
it follows that $\pi_2(\phi(x))$ is never equal to $m(\pi_1(\phi(x)))$.
By the definition of $\psi$,
this means that $\pi_2(\phi(x))\neq m(\psi(\iota(x)))$ for all $x\in T$.
Therefore the image of $\alpha$ is contained in $\bigcomp$.

The CIP for $\bigcomp$ tells us that $\alpha$ extends to a map
$\beta:\complex^n\to\bigcomp$, as in the following diagram.
(The map $\pi_2:\bigcomp\to\complex$ is the restriction to $\bigcomp$
of the projection of $\complex^n\times\complex$ onto the last coordinate.
The use of $\pi_2$ for two different projection maps should not cause any confusion,
as the domain is always clear from the context.)

$$  
\xymatrix{
    \complex & \complex^{n+1}\setminus\Gamma_{m\circ\psi} \ar[l]_-{\pi_2} &
    T \ar[l]_-\alpha \ar@{^{(}->}[dd]_\iota \ar[r]^-\phi
        & (X\times\complex)\setminus\Gamma_m \ar[dd]^{\pi_1} \ar[r]^-{\pi_2}
        & \complex \\
    \\
    && \complex^n \ar@{.>}[uul]^\beta \ar@{.>}[uur]^\mu \ar[r]^\psi
        & X \ar[r]^m & \projspace{1}}
$$

Finally, we can define $\mu:\complex^n\to\littlecomp$ by
$$\mu(x)=(\psi(x),\pi_2(\beta(x))).$$
Since $\pi_2(\beta(x))$ can never equal $m(\psi(x))$,
we see that the image of $\mu$ is indeed contained in $\littlecomp$.
And from the definitions,
the fact that $\beta$ is an extension of $\alpha$
implies that $\mu$ is an extension of $\phi$.
\end{proof}

\begin{remark} \label{rem:elliptic}
As mentioned in Remark~\ref{rem:section} above,
the projection map from $\littlecomp$ onto the first factor
is an elliptic submersion, in the case where $m$ can be written
as $f+1/g$.
To prove this, it is necessary to construct a dominating fibre spray
(\cite[page 24]{Forstneric-Larusson-2011}).
This can be done using the function $\phi$
defined by equation (\ref{eq:BLmap}) in the proof of 
Proposition~\ref{prop:Cn} above.
As previously, there is no loss of generality in assuming $f=0$.
Denoting points of $X\times\complex$ by $(x,y)$
with $x\in X$ and $y\in\complex$,
define a map $s:(\littlecomp)\times\complex\to \littlecomp$
by $$s(x,y,t)=(x,ye^{tg(x)}-\phi(g(x),t)).$$
The verification that the image of $s$ is indeed contained in $\littlecomp$,
and that $s$ is a dominating fibre spray, is routine.
\end{remark}


\end{document}